\documentclass{article}
\usepackage{fullpage}
\usepackage[utf8]{inputenc}
\usepackage{amsmath, amsfonts,amsthm, bm}
\usepackage{mathtools}
\usepackage{thm-restate}
\usepackage[dvipsnames]{xcolor}
\usepackage{caption}
\usepackage{hyperref}
\hypersetup{
    colorlinks=true,       
    linkcolor=blue,        
    citecolor=red,         
    filecolor=magenta,     
    urlcolor=cyan,         
    linktocpage=true
}
\usepackage[capitalise,noabbrev]{cleveref}

\usepackage{tikz}
\usetikzlibrary{arrows,arrows.meta, positioning,calc,shapes.geometric, math}
\definecolor{greenish}{RGB}{27,158,119}
\definecolor{MyOrange}{RGB}{217,95,2}
\definecolor{MyPurple}{RGB}{117,112,179}

\usepackage{setspace}
\usepackage[sort]{cite}

\theoremstyle{plain}
\newtheorem{thm}{Theorem}
\newtheorem{lem}[thm]{Lemma}

\newtheorem{cor}[thm]{Corollary}

\theoremstyle{definition}
\newtheorem{defn}[thm]{Definition}

\newcommand{\p}[1]{p_{(#1)}}
\usepackage{authblk}

\title{A Note on Mixed Cages of Girth 5}
\author[1]{Gabriela Araujo-Pardo}
\author[2]{Lydia Mirabel Mendoza-Cadena}
\affil[1]{Instituto de Matemáticas, Universidad Nacional Autónoma de México, Campus Juriquilla, Querétaro, Mexico.}
\affil[2]{Center for Mathematical Modeling, Universidad de Chile, Santiago, Chile.}
\date{ }

\begin{document}

\maketitle

\begin{abstract}
    \medskip
    A \emph{mixed regular graph} is a graph where every vertex has $z$ incoming arcs, $z$ outgoing arcs, and $r$ edges; furthermore, if it has girth $g$, we say that the graph is a \emph{$[z,r;g]$-mixed graph}. A \emph{$[z,r;g]$-mixed cage} is a $[z,r;g]$-mixed graph with the smallest possible order.
    
    In this note, we give a family of $[z,q;5]$-mixed graphs for $q\geq 7$ power of prime and $q-1\leq 4z+R$ with $z\geq 1$ and $R \in \{1,\ldots,5\}$. This provides better upper bounds on the order of mixed cages until this moment. 
    
    \textbf{Keywords:} mixed cages; girth; projective plane; elliptic semiplane
\end{abstract}

\section{Introduction \label{sec:intro}}
A \emph{mixed graph} is defined as a simple\footnote{No parallel arcs or edges are permitted, and neither can there be a parallel arc and edge.} graph $G=(V;E\cup A)$ where $V(G)$ represents the set of vertices, $E(G)$ designates the set of edges, and $A(G)$ is the set of arcs. We denote an edge connecting $u$ and $v$ as $uv$ and an arc directed from $u$ to $v$ as $(u,v)$ to prevent ambiguity. We consider $u$ and $v$ to be \emph{edge-adjacent} (\emph{arc-adjacent}) if there is an edge (arc) connecting them. Walks, paths, and cycles are represented by a series of vertices $(v_0, v_1, \dots, v_n)$, where each pair $v_i$ and $v_{i+1}$ is linked by either the arc $(v_i,v_{i+1})$ or the edge $v_iv_{i+1}$. The length of the shortest cycle of the graph is called the \emph{girth}.

If every vertex has $z$ incoming arcs, $z$ outgoing arcs, and $r$ edges, the graph is referred to as a \emph{mixed regular graph}. Furthermore, if it has girth $g$, we say that the graph is a \emph{$[z,r;g]$-mixed graph}. A \emph{$[z,r;g]$-mixed cage} is a $[z,r;g]$-mixed graph that has the smallest possible order, where the \emph{order} of a graph is the number of vertices. The minimum order for $[z,r;g]$-mixed graphs is denoted as $n[z,r;g]$. These graphs were introduced by Araujo-Pardo,  Hernández-Cruz, and Montellano-Ballesteros~\cite{araujo2019mixed}.

For a given set of nodes $X \subseteq V$, the graph $G[X]$ arises from removing all the vertices $V \setminus X$ of $G$ with the arcs and edges adjacent to such set.

Recent work on mixed cages is due to Araujo-Pardo, Hernández-Cruz, and Montellano-Ballesteros~\cite{araujo2019mixed}, Araujo-Pardo, {De la Cruz}, and González-Moreno~\cite{araujo2022monotonicity}, Exoo~\cite{exoo2023mixed,exoo2023amixedgraphachieving}, Araujo-Pardo and Mendoza-Cadena~\cite{araujo2024mixedcagesgirth6}, and Jajcayov\'a and Jajcay ~\cite{jajcajova2024totallyregular}.

In this note we give a new family of mixed graphs of girth five that improves the upper bound given in \cite{araujo2022monotonicity}. In fact,  these families were built using the incidence graph of an elliptic semiplane of type $C$. On the other hand, the new families presented here use the incidence graph of an elliptic semiplane of type $L$. Similar constructions were given in \cite{abajo2019improving} to obtain regular graphs of girth $5$ that attain the minimum upper bounds given for regular graphs up to this moment. 

Our constructions are inspired by those provided in \cite{abajo2019improving}. In contrast, the previous constructions given in~\cite{araujo2022monotonicity} were based on the constructions outlined in \cite{araujo2012familiesofsmall}. Moreover, for $q$ power of prime and $q-1\leq 4z+R$ with $R=\{1,\ldots,5\}$, we construct $[z,q;5]$-graphs of order $2q^2-2$, while in \cite{araujo2022monotonicity} the authors construct $[z,q;5]$-graphs of order $2q^2$ only for $q$ prime (not prime power).

Finally, it is important to highlight the significance of the families discussed here, as well as the families constructed in \cite{araujo2024mixedcagesgirth6, araujo2022monotonicity}. We present an infinite family of graphs with small order. Notably, our graphs provide improved lower bounds for any prime power $q \geq 7$ and $z \geq 1$, surpassing previous results.
\section{Preliminaries \label{sec:preliminaries}}
To be self-contained, we describe in detail three graphs that are obtained from projective planes of order $q$. For an introduction, see e.g. \cite{balbuena2008incidence, van2001course, kiss2019finite, godsil2013algebraic}.

\paragraph{Incidence graph of the projective plane on the Galois field.}  Let $q$ be a prime power, and let $\mathbb{F}_q$ be the Galois field of order $q$. 
We describe the projective plane $PG(2; q)$.
Let $L_\infty$ and $P_\infty$ be the incident line and point in the projective plane, called the infinity line and infinity point, respectively.
The classes of lines (points)  are denoted by $L_i$ ($P_i$)  for $i \in \mathbb{F}_q$.
Each line $L$ has $q+1$ points and each point $P$ is incident to $q+1$ lines. 
Any line in $L_m$ described by $y = m x +b$  is denoted by $[m,b]$, and any point in $P_x$ is denoted by $(x,y)$. 
For each $i \in  \mathbb{F}_q$, the set of points incident to $L_i$ is $\{ (i,j) \mid j \in  \mathbb{F}_q\}$ and the set of lines incident to $P_i$ is $\{ [i,j] \mid j \in   \mathbb{F}_q \}$.
Finally, for $m,b \in  \mathbb{F}_q$ the line $[m,b]$ is incident to all points $(x,y)$ such that $y = m x +b$ holds for $x,y \in  \mathbb{F}_q$.

The projective plane $PG(2; q)$ has an incidence graph associated that we denote by $G_{(2,q)}$. Each line and point has a unique associated vertex $v \in V(G_{(2,q)}) $. 
Similarly, $uv \in E(G_{(2,q)})$ if and only if  their associated line and point are incident in  $PG(2; q)$. For an easy reading, we may refer to ``line $[m,b]$'' (point $(x,y)$) instead of ``vertex associated to the line $[m,b]$'' (point $(x,y)$). Note that $G_{(2,q)}$ has order $2q^2 + 2q + 2$, diameter 3 and girth 6. In fact, it is known that this incidence graph is a $[0,q;6]$-mixed cage (undirected cage) that attains the Moore bound.

Let us describe the incidence graph of an elliptic semiplane of type $L$. This graph is derived from the projective plane. We select a line and remove it along with all of its incident points. Next, we choose another point from the remaining points and delete it, along with all its incident lines. In particular, picking line $L_{0}$ and point $P_{0}$, we obtain the following graph.

\begin{defn}[Graph $G_q$]
    Let $\mathbb{F}_q^* =  \mathbb{F}_q \setminus \{ 0 \}$. 
    For a fixed $b,y \in \mathbb{F}_q$, let us denote the sets of vertices $\mathcal{L}_b \coloneq \{ [m,b] \mid m \in \mathbb{F}^*_q\}$,  $\mathcal{P}_y \coloneq \{ (x,y) \mid x \in \mathbb{F}^*_q \}$,  $\mathcal{L}_\infty \coloneq \{ L_i \mid i \in \mathbb{F}^*_q \}$, and $\mathcal{P}_\infty \coloneq \{ P_i \mid i\in \mathbb{F}^*_q\}$.
    Finally, we define $G_q \coloneq G_{(2,q)}[ \{ \mathcal{L}_b \mid b \in  \mathbb{F}_q\} \cup \{\mathcal{P}_y  \mid y \in \mathbb{F}_q \} \cup \{ \mathcal{L}_\infty , \mathcal{P}_\infty \}]$.
\end{defn}

This graph has $2(q-1)(q+1)$ nodes and girth 6.

\paragraph{Primitive element.}  A \emph{primitive element} $\xi \in \mathbb{F}_q^*$ is a multiplicative generator of the field $\mathbb{F}_q$, such that the non-zero elements of  $\mathbb{F}_q$ can be written as $\xi^i$ for some $i\in \{ 0, \dots, q-2 \} $, that is,  $ \mathbb{F}_q = \{  0, \xi^0, \xi^1, \xi^2, \dots, \xi^{q-2}\}$. This is due to the isomorphism between the multiplicative group $\mathbb{F}^*_q$ and the additive group $\mathbb{Z}_{q-1}$, as described by Abajo, Balbuena, Bendala, and Marcote~\cite{abajo2019improving}. For example, a primitive element of the Galois Field of order 8 is $\alpha$, and a primitive element of $\mathbb{F}_7$ is 2. See e.g.~\cite{jones1998elementary}.
 
\paragraph{The circulant digraph $\overset{\bm{\rightarrow}}{\bm{C}}_{\bm{q}} \bm{{(i_1, \dots, i_k)}}$.} 
Consider the field $\mathbb{Z}_q$ and $i_1, \dots, i_k \in \mathbb{Z}_q$. A \emph{circulant digraph} $\overset{\rightarrow}{C}_q(i_1, \dots, i_k)$ has a vertex $v_i$ for each $i \in \{ 0, 1, \dots, q -1 \}$. There exists an arc $(v_{a},v_{b})$ if and only if $b \equiv  a + i \mod q$ for some $i \in i_1, \dots, i_k$. It is known that if $q = z(g-1) + 1$, then $\overset{\rightarrow}{C}_q( 1, 2, \dots, z)$ is a $[z,0;g]$-mixed graph (see e.g.~\cite{behzad1970minimal, araujo2009dicages}). 

\section{A family of mixed graphs with girth 5\label{sec:family_girth_5}}
For the Galois Field $\mathbb{F}_q$ for $q\geq 7$ a prime power, we consider the multiplicative group. Let $\xi \in \mathbb{F}^*_q$ be a primitive element. Consider the graph $G_q$, which has girth 6 and $q^2 - 1$ nodes. We describe a mixed graph of girth 5 by adding a circulant to the vertices arranged according to their second coordinate.

For a fixed $y \in \mathbb{F}_q$, we add circulants to $\mathcal{P}_y$ and to $\mathcal{L}_\infty$,  where all the computations are made modulo $q$. For a jump of length $i$, we add arcs $\big( (x,y),(x\xi^i,y) \big)$ and $\big( L_x, L_{x\xi^i} \big)$ for each $x \in \mathbb{F}^*_q$. Similarly for lines and the infinity points, for fixed $b \in \mathbb{F}_q$, we add arcs $\big( [m,b],[m/\xi^i,b] \big)$ and $\big( P_m, P_{m/\xi^i} \big)$ for each $m \in \mathbb{F}^*_q$. In all cases, we add arcs for $i \in \{ 1,2, \dots, k\}$ where $k$ is such that $q-1 = 4k + R$ for some $1\leq R\leq 5$. Let us call this resulting graph $H_q$; see Figure~\ref{fig:example_all_circulants_q8} for an example.

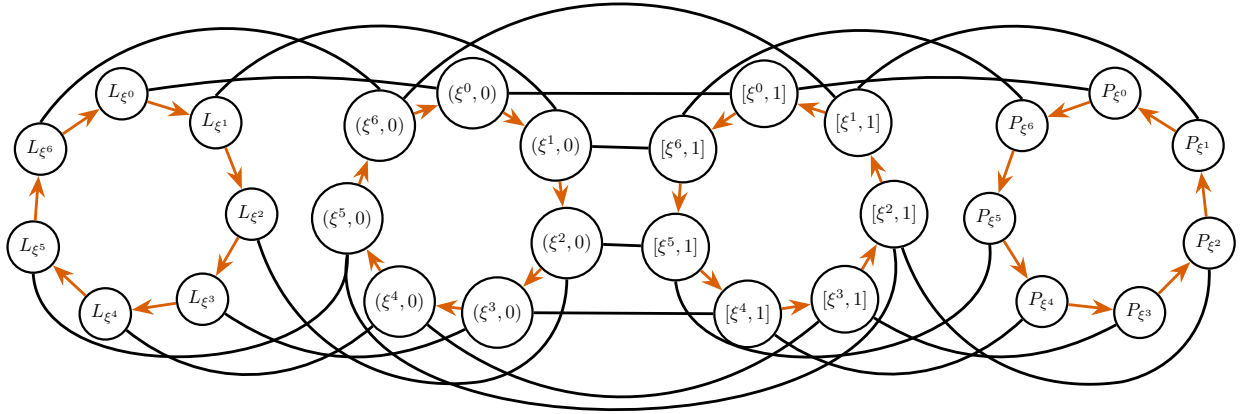
\begin{figure}
    \centering
    \resizebox{\textwidth}{!}{
    \begin{tikzpicture}[scale=.75,node distance=7em, 
        myArc/.style={draw,line width = 1.5pt, -{Stealth[length=4mm]},MyOrange}, 
        myEdge/.style={line width = 1.5pt},     
        stateColor/.style={circle,  minimum size=2.2em, draw, line width = 1.2pt}]   
        \node[minimum size={12em},regular polygon,regular polygon sides=7,rotate=45] at (1,1) (7-gonPOINTS) {};
        \foreach \p [count = \c from 0] in {7,...,1}{%
            \node[stateColor] at (7-gonPOINTS.corner \p) (point\c) {$(\xi^\c,0)$};};
        \foreach \i in {0, ..., 6}{
            \tikzmath{integer \j;
            \j = Mod(\i+1,7);}
            \draw[myArc] (point\i) to (point\j);
            }   
        \node[minimum size={12em},regular polygon,regular polygon sides=7,rotate=60] at (1+8,1) (7-gon) {};
        \foreach \p [count = \c from 0] in {7,...,1}{%
            \node[stateColor] at (7-gon.corner \p) (line\c) {$[\xi^\c,1]$};};
        \foreach \i in {0, ..., 6}{
            \tikzmath{integer \j;
            \j = Mod(\i-1,7);}
            \draw[myArc] (line\i) to (line\j);
            }   
    
        \node[minimum size={12em},regular polygon,regular polygon sides=7,rotate=45] at (1+16,1) (7-gonPOINTSinfty) {};
        \foreach \p [count = \c from 0] in {7,...,1}{%
            \node[stateColor] at (7-gonPOINTSinfty.corner \p) (inftyPoint\c) {$P_{\xi^\c}$};};
        \foreach \i in {0, ..., 6}{
            \tikzmath{integer \j;
            \j = Mod(\i-1,7);}
            \draw[myArc] (inftyPoint\i) to (inftyPoint\j);
            }

        \node[minimum size={12em},regular polygon,regular polygon sides=7,rotate=60] at (1-8,1) (7-goninfty) {};
        \foreach \p [count = \c from 0] in {7,...,1}{%
            \node[stateColor] at (7-goninfty.corner \p) (inftyLine\c) {$L_{\xi^\c}$};};
        \foreach \i in {0, ..., 6}{
            \tikzmath{integer \j;
            \j = Mod(\i+1,7);}
            \draw[myArc] (inftyLine\i) to (inftyLine\j);
            }
    
         \foreach \i \j \bending in {0/0/0,6/1/48,5.south/2.south/-100,4/3/-40,3/4/0,2/5/0,1/6/0}{
            \draw[myEdge, bend left=\bending] (point\i) to (line\j); 
            }
        \foreach \i \bending in {0/8,1.north/50,3/-30,4/-40,5.south/-90,6.north/50}{
            \draw[myEdge, bend left=\bending] (inftyLine\i) to (point\i); 
            \draw[myEdge, bend left=\bending] (line\i) to (inftyPoint\i);
            } 
            \draw[myEdge, bend left=-40] (inftyLine2) to ($(point4) + (2,-2)$)  to (point2.south); 
            \draw[myEdge, bend left=-40] (line2) to ($(inftyPoint4) + (2,-2)$)  to (inftyPoint2.south);
    \end{tikzpicture}
    
    }
    \caption{Some vertices and their adjacencies of graph $H_8$ from the Galois field of order 8. Here, $k= 1$ as $7 = 5 \times 1 +2$. The figure shows parts $\mathcal{L_\infty}, \mathcal{P}_0, \mathcal{L}_1$ and $\mathcal{P}_\infty$.}
    \label{fig:example_all_circulants_q8}
\end{figure}

Note that the circulants define a partition on the nodes as follows. Points are arranged according to their second coordinate, hence we have a partition with $2(q+1)$ parts, each part containing $q-1$ elements. That is, we have parts $\mathcal{P}_y$ for a fixed  $y \in \mathbb{F}_q$,  $\mathcal{L}_b$ for a fixed $b\in \mathbb{F}_q$, $\mathcal{L}_\infty$ and $\mathcal{P}_\infty$. Using this partition, it is not difficult to see that for each point $(x,y)$ (which belongs to part $\mathcal{P}_y$ for a fixed $y \in \mathbb{F}^*_q$), there exists a unique line $[\frac{y-b}{x},b]$ in part $\mathcal{L}_b$ for $b \in \mathbb{F}_q$; this is well-defined as $x\in \mathbb{F}^*_q$. Hence, for any fixed pair $y,b \in \mathbb{F}_q$ there is a matching between parts $\mathcal{P}_y$ and $\mathcal{L}_b$. The same hold for parts $\mathcal{P}_y$ for a fixed  $y \in \mathbb{F}_q$ and $\mathcal{L}_\infty$, and for parts $\mathcal{L}_b$ for a fixed $b\in \mathbb{F}_q$ and $\mathcal{P}_\infty$. Note that parts $\mathcal{L}_\infty$ and $\mathcal{P}_\infty$ share no edge.

\begin{lem}
    $H_q$ has girth 5.
\end{lem}
\begin{proof}
    First, note that as $H_q$ arose from the undirected graph $G_q$, then all undirected cycles have length at most 6. For fixed $y\in \mathbb{F}_q$, on each set of $(q-1)$ vertices we added a circulant $\overset{\rightarrow}{C}_{q-1}(1, \dots, k)$ for $q-1 = 4k + R$ for some $R \in \{ 1, \dots, 5\}$, and hence any directed cycle has length at least 5.
    We verify that any mixed cycle has a length at least 5.

    Based on the partition and the matching among parts discussed earlier, it is clear that no mixed cycles of lengths 2 and 3 exist.

    Suppose that there is a mixed cycle of length 4. We analyze all possible cases based on the sequence of lines and points that the cycle may consist of.
    Note that the case where it consists of a line, a point, a line and a point, is not possible as this can only happen with edges.
    If the mixed cycle consists on a point $(x,y)$ followed by three lines, then the lines should have the same second coordinate $b$ to have an arc among them, that is, the three lines belong to a part $\mathcal{L}_b$ for a fixed $b\in \mathbb{F}_q$ (or to the part $\mathcal{L}_\infty$). But then such a mixed cycle cannot exist as the point $(x,y)$ is only adjacent to exactly one of the lines by the matching between parts $\mathcal{P}_y$ and  $\mathcal{L}_b$ (or $\mathcal{L}_\infty$). Similarly, a point $P_m$ followed by three lines cannot exist.

    Let us analyze the case when the mixed cycle consists on two arc-adjacent points $(x,y)$ and $(x',y)$ followed by two -arc-adjacent lines $[m,b]$ and $[m',b]$, that is, the mixed cycle can be written as  $\big(  (x,mx+b), [m,b], [m/\xi^i,b], (x',(m/\xi^i)x' + b) \big)$ for some $i \in \{ 1, \dots, k\}$. As the two points are arc-adjacent, then their second coordinate is the same, and $mx+b = (m/\xi^i)x' + b $, or equivalently $x\xi^i = x'$. This implies that the four nodes do not form a cycle: there is an arc $([m,b], [m/\xi^i,b])$, and an arc $\big( (x,y), (x', y) \big)$, instead of $\big( (x',y), (x, y) \big)$.

    The case where the mixed cycle consists of two arc-adjacent points $(x,y)$ and $(x',y)$ followed by two -arc-adjacent lines $L_{x}$ and $L_{x'}$ cannot happen as if there is an arc $\big( (x,y), (x', y) \big)$, then there is an arc $\big( L_{x}, L_{x'} \big)$. Similarly, for two arc-adjacent points $P_m, P_{m'}$ and two arc-adjacent lines $[m,b]$ and $[m',b]$. By the case analysis, we conclude that there is no mixed cycle of length 4.

    Finally, we exhibit a mixed cycle of length 5. Consider $\big( (x,y), [m,b], (x/\xi^1,y'), (x,y'), L_{x} \big)$ for $x,m \in \mathbb{F}^*_q$, $y,y',b \in \mathbb{F}_q$, in particular, if $\mathbb{F}_q = \{ 0,1, \xi^1, \xi^2, \dots, \xi^{q-2} \}$, $\big( (1,0), [m,b], (\xi^{q-2},m\xi^{q-2} +b ), (1,m\xi^{q-2} +b), L_{1} \big)$.    
\end{proof}

\begin{cor}
    For all prime power $q$ and $k$ such that $q-1 = 4k + R$ for some $R\in \{ 1, \dots 5\}$, the graph $H_q$ is a $[k,q,5]$-mixed graph.
\end{cor}

\begin{cor}
    Let  $q$ be a prime power and let $k$ be a positive integer such that $q-1 = 4k + R$ for some $R\in \{ 1, \dots, 5\}$.
    Let $n[k,q;5]$ be the minimum order of a $n[k,q;5]$-mixed cage. Then, $n[k,q;5]\leq 2q^2-2$.
\end{cor}

As we discussed earlier, the importance of these constructions lies in their ability to offer broad, general frameworks on the construction of mixed graphs of given girth. They serve as a foundation for understanding and applying various concepts effectively. To conclude, let us compare the upper bounds given in this paper with the lower bound for any $z\geq 1$ provided in \cite {araujo2024mixedcagesgirth6}. This lower bound is a generalization of the lower bound for  $z=1$, known as the AHM bound given in~\cite{araujo2019mixed}. We have the following inequalities: $q^2+q+4z+1\leq n[z,q;5] \leq 2q^2-2$ for $q-1=4k+R$ and $R\in \{ 1, \dots, 5\}$.

\subsubsection*{Acknowledgments}
G. Araujo-Pardo was supported by PAPIIT-UNAM-M{\' e}xico IN113324.

G. Araujo-Pardo and L.M. Mendoza-Cadena was supported CONAHCyT: CBF2023-2020-552 M{\' e}xico.

L. M. Mendoza-Cadena was supported by the Lend\"ulet Programme of the Hungarian Academy of Sciences -- grant number LP2021-1/2021, by the Ministry of Innovation and Technology of Hungary from the National Research, Development and Innovation Fund -- grant ELTE TKP 2021-NKTA-62.


\bibliographystyle{abbrv}
\bibliography{cages}

\begin{thebibliography}{10}

\bibitem{abajo2019improving}
E.~Abajo, C.~Balbuena, M.~Bendala, and X.~Marcote.
\newblock Improving bounds on the order of regular graphs of girth 5.
\newblock {\em Discrete Mathematics}, 342(10):2900--2910, 2019.
\newblock Algebraic and Extremal Graph Theory.

\bibitem{araujo2012familiesofsmall}
G.~Abreu~Marien, Araujo-Pardo, C.~Balbuena, and D.~Labate.
\newblock Families of small regular graphs of girth 5.
\newblock {\em Discrete Mathematics}, 312:2832--2842, 2012.

\bibitem{araujo2009dicages}
G.~Araujo-Pardo, C.~Balbuena, and M.~Olsen.
\newblock On (k, g; l)-dicages.
\newblock {\em Ars Combinatoria}, 92:289--301, 2009.

\bibitem{araujo2022monotonicity}
G.~Araujo-Pardo, C.~{De la Cruz}, and D.~González-Moreno.
\newblock Mixed cages: Monotonicity, connectivity and upper bounds.
\newblock {\em Discrete Mathematics}, 345(5):112792, 2022.

\bibitem{araujo2019mixed}
G.~Araujo-Pardo, C.~Hernández-Cruz, and J.~J. Montellano-Ballesteros.
\newblock Mixed cages.
\newblock {\em Graphs and Combinatorics}, 35(5):989--999, 2019.

\bibitem{araujo2024mixedcagesgirth6}
G.~Araujo-Pardo and L.~M. Mendoza-Cadena.
\newblock On mixed cages of girth 6, 2024.

\bibitem{balbuena2008incidence}
C.~Balbuena.
\newblock Incidence matrices of projective planes and of some regular bipartite
  graphs of girth 6 with few vertices.
\newblock {\em SIAM Journal on Discrete Mathematics}, 22(4):1351--1363, 2008.

\bibitem{behzad1970minimal}
M.~Behzad, G.~Chartrand, and C.~E. Wall.
\newblock On minimal regular digraphs with given girth.
\newblock {\em Fundamenta Mathematicae}, 69(3):227--231, 1970.

\bibitem{exoo2023amixedgraphachieving}
G.~Exoo.
\newblock A mixed graph achieving a moore-like bound, 2023.

\bibitem{exoo2023mixed}
G.~Exoo.
\newblock {On Mixed Cages}.
\newblock {\em {Discrete Mathematics \& Theoretical Computer Science}}, {vol.
  25:2 }, Nov. 2023.

\bibitem{jones1998elementary}
J.~M.~J. Gareth A.~Jones.
\newblock {\em Elementary Number Theory}.
\newblock Springer London, 1998.

\bibitem{godsil2013algebraic}
C.~Godsil and G.~F. Royle.
\newblock {\em Algebraic graph theory}, volume 207.
\newblock Springer Science \& Business Media, 2013.

\bibitem{jajcajova2024totallyregular}
T.~Jajcayov\'a and R.~Jajcay.
\newblock Totally regular mixed graphs constructed from the $cd(n,q)$ graphs of
  lazebnik, ustimenko and woldar.
\newblock {\em The Art of Discrete and Applied Mathematics}, 2025.
\newblock Algebraic and Extremal Graph Theory.

\bibitem{kiss2019finite}
G.~Kiss and T.~Szonyi.
\newblock {\em Finite geometries}.
\newblock Chapman and Hall/CRC, 2019.

\bibitem{van2001course}
J.~H. Van~Lint and R.~M. Wilson.
\newblock {\em A course in combinatorics}.
\newblock Cambridge university press, 2001.

\end{thebibliography}
\end{document}